\documentclass[11pt]{amsart}
\numberwithin{equation}{section}
\usepackage{graphicx}
\usepackage{amsmath}
\usepackage{amsthm}
\usepackage{amssymb}
\usepackage{amsfonts}
\usepackage{mathtools}
\usepackage{mathbbol}
\usepackage{caption}
\usepackage{spverbatim}
\usepackage{subcaption}
\usepackage{textcomp}
\usepackage{booktabs}
\usepackage{array}
\usepackage{tabularray}
\UseTblrLibrary{booktabs}
\usepackage[english]{babel}
\usepackage[hidelinks]{hyperref}
\usepackage{enumitem}
\setlist[enumerate]{
  label={\upshape(\roman*)},
}
\usepackage{url} %breaks url

\usepackage{breakurl}
\usepackage{xcolor}
\hypersetup{breaklinks=true}

\usepackage{cite}
\newtheorem{theorem}{Theorem}[section]

\newtheorem{lemma}[theorem]{Lemma}
\newtheorem{toptheorem}{Theorem}

\newtheorem{topcorollary}[toptheorem]{Corollary}

\newtheorem{proposition}[theorem]{Proposition}

\newcommand{\bbz}{\mathbb{Z}}

\newcommand{\zeros}{\mathcal{V}}

\newcommand{\defn}[1]{\emph{#1}}
\newcommand{\eq}[1]{(\ref{#1})}
\newcommand{\tab}[1]{Table~\ref{#1}}
\newcommand{\deff}[1]{Definition~\ref{#1}}
\newcommand{\thm}[1]{Theorem~\ref{#1}}

\newcommand{\lm}[1]{Lemma~\ref{#1}}
\newcommand{\prop}[1]{Proposition~\ref{#1}}
\newcommand{\coro}[1]{Corollary~\ref{#1}}
\newcommand{\rmk}[1]{Remark~\ref{#1}}

\newcommand{\eg}[1]{Example~\ref{#1}}
\newcommand{\sect}[1]{Section~\ref{#1}}
\theoremstyle{definition}
\newtheorem{definition}[theorem]{Definition}
\newtheorem{example}[theorem]{Example}

\newtheorem{remark}[theorem]{Remark}

\def \g {\mathfrak{g}}
\renewcommand{\Im}{\mathrm{Im}}

\newcommand{\End}{\ensuremath{\mathrm{End}}}
\newcommand{\lcm}{\ensuremath{\mathrm{lcm}}}
\newcommand{\irr}{\mathrm{Irr}}

\newcommand{\Ext}{\ensuremath{\mathrm{Ext}}}
\newcommand{\opp}{\ensuremath{\mathrm{op}}}

\newcommand{\stcond}{(FinHopf)~}
\newcommand{\cdk}[1]{$(\mathbf{CH_{#1}})$~}
\newcommand{\infcy}{\ensuremath{\mathbb{Z}}} %infinite cyclic group
\newcommand{\Zn}[1]{\ensuremath{\mathbb{Z}/#1\mathbb{Z}}}
\newcommand{\field}{\ensuremath{\mathbb{k}}}
\newcommand{\fieldx}{\ensuremath{\mathbb{k}^{\times}}}
\newcommand{\mfield}{\mathbb{G}_{\mathrm{m}}}
\newcommand{\cexta}{\ensuremath{N}}
\newcommand{\cexte}{\ensuremath{\hat{G}}}
\newcommand{\cextb}{\ensuremath{Q}}

\newcommand{\cmmtx}{\ensuremath{M}}
\newcommand{\cmmtxij}{\ensuremath{m}}
\newcommand{\orderu}{\ensuremath{d}}

\newcommand{\halg}{\ensuremath{H}} 
\newcommand{\nalg}{\ensuremath{R}} 
\newcommand{\fdh}[1]{\ensuremath{#1^{\circ}}} 
\newcommand{\gpfdh}[1]{\ensuremath{G(\fdh{#1})}}
\newcommand{\calg}{\ensuremath{C}} 
\newcommand{\maxi}{\mathbf{m}}
\newcommand{\runit}{\ensuremath{\xi}} 
\newcommand{\idc}{\ensuremath{\maxi_{\bar{\epsilon}}}} 
\newcommand{\orbidc}{\ensuremath{\mathcal{O}}}
\newcommand{\bgp}{\ensuremath{\Lambda}} 
\newcommand{\agp}{\ensuremath{\Delta}} 
\newcommand{\egp}{\ensuremath{\Sigma}} 
\newcommand{\cgp}{\ensuremath{\Omega}} 
\newcommand{\tccl}{\ensuremath{\beta}} 
\newcommand{\cocyl}[2]{\ensuremath{Z^2(#1,#2)}}
\newcommand{\cobnd}[2]{\ensuremath{B^2(#1,#2)}}
\newcommand{\cohom}[2]{\ensuremath{H^2(#1,#2)}}
\newcommand{\mtx}[2]{\ensuremath{M_{#1}(#2)}}
\newcommand{\abs}[1]{\left|#1\right|}
\newcommand{\ordergp}[2]{\abs{#1}_{#2}}
\newcommand{\stabb}{\mathrm{Stab}}
\newcommand{\stab}[1]{\mathrm{Stab}_{G_0}(#1)}
\newcommand{\tensork}{\otimes_{\field}}
\newcommand{\ann}[2]{\mathrm{ann}_{#1}(#2)}
\newcommand{\sd}[1]{\mathrm{Sd}(#1)}
\newcommand{\maxspec}[1]{\mathrm{MaxSpec}(#1)}
\newcommand{\tr}{\mathrm{tr}}
\newcommand{\trreg}{\tr_{\mathrm{reg}}}
\newcommand{\intord}{\mathrm{io}}
\newcommand{\generate}[1]{\ensuremath{\langle #1 \rangle}}
\newcommand{\D}{D(R/C,\tr)}
\newcommand{\Dn}[1]{D_{#1}(R/C,\tr)}
\newcommand{\MDn}[1]{MD_{#1}(R/C,\tr)}
\newcommand{\Dnh}[1]{D_{#1}(H/C,\tr)}
\newcommand{\autk}{\mathrm{Aut}_{\field-\mathrm{Alg}}}
\newcommand{\algk}{\mathrm{Alg}_{\field}}
\newcommand{\algkk}[1]{\algk(#1,\field)}
\newcommand{\lw}{\mathrm{W}_{\mathrm{l}}}
\newcommand{\rw}{\mathrm{W}_{\mathrm{r}}}
\newcommand{\chark}{\mathrm{char}\,}
\newcommand{\incl}{\hookrightarrow}
\newcommand{\algtr}[1]{(#1,\calg,\tr)}
\title{The Lowest Discriminant Ideal of Central Extensions of abelian Groups}

\author[Zhongkai Mi]{Zhongkai Mi}
\address{Shanghai Center for Mathematical Sciences \\
Fudan University \\
Shanghai, 200438 \\
People's Republic of China}
\email{zhongkai\_mi@fudan.edu.cn}
\thanks{The research  of the author has been supported by NSF grant DMS--2200762.}
\date{}

\begin{document}
\begin{abstract}
In a previous joint paper with Wu and Yakimov, we gave an explicit description of the lowest discriminant ideal in Cayley-Hamilton Hopf algebras $(\halg,\calg,\tr)$ with basic identity fiber, i.e. all irreducible representations over the kernel of the counit of the central Hopf subalgebra $\calg$ are one-dimensional. 
We first show in this work that the zero set of lowest discriminant ideal of the group algebra of any central extension of an arbitrary finite group is the whole maximal spectrum of $\calg$. Then we specialize to central extensions of abelian groups and study the fiber algebras and orbits under automorphism groups in this zero set of the lowest discriminant ideal. An example is given, in which isomorphisms of fiber algebras do not always lift to automorphisms of $\halg$.

%In a previous joint paper with Wu and Yakimov, we gave an explicit description of the lowest discriminant ideal in a Cayley-Hamilton Hopf algebra $(\halg,\calg,\tr)$ of degree $d$ over an algebraically closed field $\field$, $\chark\field\notin[1,d]$ with basic identity fiber, i.e. all irreducible representations over the kernel of the counit of the central Hopf subalgebra $\calg$ are one-dimensional. Using results developed in that paper, we compute relevant quantities associated with irreducible representations to explicitly describe the zero set of the lowest discriminant ideal in the group algebra of a central extension of the product of two arbitrary finitely generated abelian groups by any finite abelian group under some conditions. Over a fixed maximal ideal of $\calg$ the representations are tensor products of representations each corresponding to a central extension of a subgroup isomorphic to the product of two cyclic groups of the same order. A description of the orbit of the identity, i.e. the kernel of the counit of $\calg$, under winding automorphisms is also given.
\end{abstract}
\maketitle

\section{Introduction}
Discriminants and  discriminant ideals (\deff{def:dist}) are defined for an algebra with trace $(\nalg,\calg,\tr)$ (\deff{def:alg_tr}). In literature $\nalg$ is always a finitely generated algebra over a field that is module finite over a central subalgebra $\calg$. Discriminants are important invariants and has been a very active area of research in noncommutative algebra in recent years. On one hand, there is work on their computation using techniques such as smash products \cite{GKM}, Poisson geometry \cite{LY}\cite{NTY1}, cluster algebras \cite{NTY2} and reflexive hulls \cite{CGWZ}. On the other hand, it has been used to study automorphism groups of noncommutative algebras \cite{CPWZ1}\cite{CPWZ2} and the Zariski cancellation problem\cite{BZ1}. The Zariski cancellation problem is about whether $A\cong B$ as algebras when $A[X]\cong B[X]$. Discriminant ideals are much more general and do not require some ideal to be principal. Currently very little is known about discriminant ideals. If $\nalg$ is a prime algebra over an algebraically closed field and $\calg$ is its center and normal, $\tr$ is a representation-theoretic trace; the zero set of the highest discriminant ideal $\Dn{h}$ (\deff{def:hl_dist}) is the complement of the Azumaya locus \cite[Main Theorem]{BY}. This result relies on a description of the zero sets of discriminant ideals in terms of representations as follows: if

\medskip
\cdk{n} $(\nalg,\calg,\tr)$ is a finitely generated Cayley-Hamilton algebra (\deff{def:CH_alg}) of degree $n$ over an algebraically closed field $\field$ and $\chark\field\notin[1,n]$,
\medskip

%\noindent there is a description of the zero sets of discriminant ideals by dimensions of irreducible representations
\noindent then \cite[Theorem 4.1(b)]{BY}:
\begin{align}
  \zeros_k&=\zeros(\Dn{k})=\zeros(\MDn{k}) \nonumber\\
  &=\bigg\{ \maxi \in \maxspec{\calg} \bigg\lvert\sd{\maxi}=\sum_{V\in \irr(R/\maxi R)} \dim(V)^2<k\bigg\}.
  \label{eq:sd}
\end{align}
where $\irr(\nalg/\maxi \nalg)$ denotes isomorphism classes of irreducible representations of $\nalg$ over $\maxi\in\maxspec{\calg}$. The lowest discriminant ideal (\deff{def:hl_dist}) is in general very hard to handle as it is the most degenerate.

If $(\halg,\calg,\tr)$ is a prime Cayley-Hamilton Hopf algebra (\deff{def:CH_alg}) then the zero set of the lowest discriminant ideal $\zeros(\Dnh{l})$ is in the complement of the unramified locus; which is a dense, open and proper subset of $\maxspec{\calg}$ \cite[Corollary 2.2]{BG2001}. However this does not hold true for group algebras as we will see in \thm{tthm:dist}.

In a Cayley-Hamilton Hopf algebra $(H,C,\tr)$, let $\idc$ denote the kernel of the counit of $\calg$. Suppose $\halg/\idc \halg$ (the identiy fiber algebra) is basic, i.e. all of its irreducible representations are one-dimensional; the level $l$ of the lowest discriminant ideal (\deff{def:hl_dist}) is the number of isomorphism classes of such representations plus one \cite[Theorem B(b)]{MWY2023}. Without assuming $\halg/\idc \halg$ is basic, this level is greater than the integral order of $\halg$ \cite[Theorem B(d)]{MWY2023}. The integeral order of $\halg$ is no greater than the size of the orbit of $\idc$ under the action of winding automorphism groups $\lw(\gpfdh{\halg})$ or $\rw(\gpfdh{\halg})$ \cite[Proof of Theorem B(d) and Theorem C(a)]{MWY2023}.% We consider a pair $(\halg,\calg)$ of group algebras arising from central extensions of finite groups. In fact this is the most general setting for cocommutative Hopf algebras in characteristic $0$ \textcolor{red}{cite}. % The zero set of $\Dnh{l}$ contains the orbit of $\idc$ under left and right winding automorphisms of $\halg$ and $\maxi\in\maxspec{\calg}$ is in the orbit if and only if its fiber is basic \cite[Theorem C(a)]{MWY2023}.  A special case in \cite{MWY2023} is a central extension of $\Zn{l}\times\Zn{l}$ by $\Zn{l}$. Here $\zeros_l=\maxspec{\calg}$, thus $l=h+1$. We first consider the most general group algebras satisfying \stcond:

%\begin{toptheorem}
%Let $\bgp$ be a finitely generated abelian group and $\agp$ be a finite abelian group. Suppose there is a central extension
%\begin{equation*}
%1\xlongrightarrow{}\agp\xlongrightarrow{}\egp\xlongrightarrow{f}\bgp\xlongrightarrow{} 1.
%\end{equation*}
%Then 
%\end{toptheorem}
We first consider group algebras of central extensions of finite groups, where the identity fiber algebra is not necessarily basic. In fact up to tensor product with a polynomial ring, this is the most general setting in characteristic zero (\prop{prop:general}) for a cocommutative pair $(H,C)$ satisfying 

\medskip
\stcond $\halg$ is a finitely generated Hopf algebra over an algebraically closed field $\field$ and a finite module over a central Hopf subalgebra $C$.% and there is a trace $\tr:\nalg\mapsto\calg$.
\medskip

\begin{toptheorem}
%Let $H$ be a finitely generated group algebra over an algebraically closed field $\field$ with $\chark\field=0$ and a finite module over a central Hopf subalgebra $C$. Then
%\begin{enumerate}
%\item there is a finitely generated group $G$ and a finitely generated abelian group $A$ in the center of $G$ such that $H=\field G$ and $C=\field A$. Let $\trreg: H\mapsto C$ be the regular trace then, $(H,C,\trreg)$ is a Cayley-Hamilton Hopf algebra of degree $d=[G:A]$.
Let
\begin{equation*}
  1\longrightarrow \cexta\overset{i}{\longrightarrow} \cexte\overset{p}\longrightarrow \cextb\longrightarrow 1
\end{equation*}
be a central extension of a finite group $\cextb$ by a finitely generated abelian group $\cexta$ and $\field$ be an algebraically closed field with $\chark\field\notin[1,\abs{\cextb}]$. Define $\halg=\field \cexte$ and $\calg=\field \cexta$.
\begin{enumerate}
\item Denote by $\trreg: H\to C$ the regular trace. Then $(H,C,\trreg)$ is a Cayley-Hamilton Hopf algebra of degree $d=\abs{\cextb}$. \label{tthm:dist:CH}
\item Suppose $\maxi\in\maxspec C$, then $H/\maxi H\cong \field_{\tccl} (\cextb)$ for some 2-cocycle $\tccl\in\cocyl{\cextb}{\fieldx}$. Consequently,  $H/\maxi H$ is semisimple and 

    \begin{equation}
      \zeros(D_k(\halg/\calg,\trreg))=\zeros(MD_k(\halg/\calg,\trreg))=\begin{cases}
			       \varnothing,\quad &k \le d,\\
			       \maxspec {C},\quad &k>d.
                             \end{cases}
    \label{eq:main_dist}
    \end{equation}
    In particular, $\zeros(\Dnh{l})=\maxspec{\calg}$.\label{tthm:dist:di}
  \item The integral order of $H$ is $1$ or equivalently its left and right homological integrals are both trivial.\label{tthm:dist:io}
\end{enumerate}
\label{tthm:dist}
\end{toptheorem}

%If the Hopf algebras $H$ and $C$ are known to be group algebras described in part (i), part (ii) holds when $\chark\field\notin [1,\abs{G:A}]$ and part (iii) is always true. We will look into the case when $H$ is the group algebra of central extensions of finitely generated abelian groups by finite abelian groups and compare with the analysis in \cite{MWY2023}.

There is the naturally arising question of whether the fiber algebra $H/\maxi H\cong \field_{\tccl} (\cextb)$ is simple. A necessary condition is that $\cextb\cong L/Z(L)$ for a group of central type $L$. A group $L$ is of \defn{central type} if it has an irreducible character $\phi$ with degree $\phi(1)=\sqrt{\abs{L:Z(L)}}$. Such a group is always solvable \cite{LY1979} and so if $\halg/\maxi \halg$ is simple for some $\maxi\in\maxspec C$ then $\cextb$ is solvable. However, $\cextb$ being solvable is by no means a sufficient condition, as an abelian $\cextb$ is always solvable but $\abs{\cextb}$ may not be a square and then $\field_{\tccl}\cextb$ can not be simple by Artin-Wedderburn Theorem. The question can be addressed by determining the number of irreducible representaions following \cite[Chapter 10.2]{KAR}.

%Note in this case $\abs{G:A}$ must also be a square by Artin-Wedderburn Theorem.

%We give a more accurate description of a pair $(\halg,\calg)$ described in part(i) of \thm{tthm:dist} to understand its homological integral and Zariski cancellation properites (\textcolor{red}{defined in}):

%\begin{toptheorem}
%Let $G$ be a finitely generated group with a central subgroup of finite index. Define $H=\field G$. Then
%\begin{enumerate}
%  \item there is a finite group $\tilde{G}$ and a finitely generated free abelian group $\bbz^n$ such that $G \cong\tilde{G}\times \bbz^n$. 
%  \item The integral order of $H$ is $1$ or equivalently its left and right homological integrals are both trivial.
%  \item $H$ is strongly cancellative.
%\end{enumerate}
%\label{tthm:ioz}
%\end{toptheorem}

%If a Hopf algebra $H$ is finite-dimensional, the integral order $\intord(H)=1$ if and only if $H$ is semisimple \cite{LS}. This is not true when $H$ is infinite-dimensional. It is easy find examples of positive global dimensions covered by \thm{tthm:dist}, e.g. the group algebra of $\cextb\times \bbz$ with $\cextb$ finite. Nevertheless, any fiber algebra is semisimple.

Next we add the condition that $\halg/\idc\halg$ basic, or equivalently $\cextb$ is an abelian group to study the fiber algebras and orbits of elements in the lowest discriminant ideal $\maxspec{\calg}$ under automorphism groups. Such Hopf algebras include the group algebra of a central extension of a finitely generated (not necessarily finite) abelian group (see \sect{sec:central-ext} for more details). In this setting, maximally stable modules under the action of the group $\gpfdh{(\halg/\idc\halg)}$ (\deff{def:max-stab}) play a central role in the description of the lowest discriminant ideal \cite[Theorem B(c)]{MWY2023}.% We also consider the action of the larger group $\gpfdh{\halg}$ through winding automorphisms $\lw$ and $\rw$.   % Define $G_0\coloneq(\halg/\idc\halg)^{\circ}$ then $G_0\cong G$ and $\stab{V}$ as the stabilizers of the action of $G_0$ on an irreducible 
%itto consider the existence of maximally stable modules (\deff{def:max-stab}) and orbits under winding automorphisms. M
\begin{topcorollary}
  Let $\cexte$ be a central extension of a finite abelian group $\cextb$ by a finitely generated abelian group $\cexta$ and $\field$ be an algebraically closed field with $\chark \field\notin[1,\abs{\cextb}]$. Define $\halg=\field\cexte$ and $\calg=\field \cexta$.
  \begin{enumerate}
  \item Let $\maxi\in\maxspec{\calg}$, $V\in\irr(\halg/\maxi\halg)$, $\dim V=n$ and $k =\abs{\cextb}/n^2$. then 
  \begin{equation*}
    \halg/\maxi\halg\cong \mtx{n}{\field}^{\oplus k}.
  \end{equation*}
  \label{tcor:basic:semi}
  \item All irreducible modules of $\halg$ are maximally stable.
  \label{tcor:basic:stab}
  \item 
   Identify $\cexta$ with a subgroup of $\cexte$. There is a subgroup $A\lhd \cexta$ such that $\cexte/A$ is abelian and
    \begin{align}
      \orbidc &=\autk(\halg,\calg)\cdot \idc=\lw(\gpfdh{\halg})\cdot \idc=\rw(\gpfdh{\halg})\cdot \idc\label{eq:orbit_one}\\
      & \cong \gpfdh{(\field ~\cextb/A)}\subset\gpfdh{\calg}.%\label{eq:orbit_two}.
    \end{align}
  \label{tcor:basic:orbit}
  \end{enumerate}
\label{tcor:basic}
\end{topcorollary}

Part\ref{tcor:basic:semi} implies that two fiber algebras are isomorphic if and only if their irreducible represenations have the same dimension.
  The orbit $\orbidc$  is the subgroup of $\gpfdh{C}\cong\maxspec{\calg}$ with basic fibers or the intersection with $\calg$ of annihilators of one-dimensional irreducible $\halg$-modules. Such transitive property may not hold in higher dimensions. In \sect{sec:central-ext} we show that in general winding automorphism groups or even $\autk(\halg,\calg)$ does not act transitively on a subset of $\maxspec{\calg}$ with isomorphic fiber algebras or a subset sharing some stronger property.

\section{Background}
In this section, we review definitions related to discriminant ideals, Cayley-Hamilton Hopf algebras and winding automorphisms of Hopf algebras.
\begin{definition}
Let $(\nalg,\calg)$ be an algebra $\nalg$ over a field $\field$ with a central subalgebra $\calg$. Following the definition in \cite{BY} a nonzero function $\tr:\nalg \mapsto\calg$ is called a \defn{trace} if it is $\calg$-linear and cyclic, i.e.
\begin{enumerate}
  \item $\tr(r_1+c r_2)=\tr(r_1)+c\tr(r_2)\quad\forall c\in\calg,\,r_1,\,r_2\in\nalg$;
  \item $\tr(r_1 r_2)=\tr(r_2 r_1)\quad r_1,\,r_2\in\nalg$.
\end{enumerate}
If such a function exits, $(\nalg,\calg,\tr)$ is called an \defn{algebra with trace}.
\label{def:alg_tr}
\end{definition}

Assuming that $\tr(1)$ has no $\infcy$-torsion, then $\nalg$ has invariant basis number (IBN), i.e. $\nalg^n\cong\nalg^m$ implies $n=m$ \cite[Proposition 13.29]{RW1}. 

\begin{example}[regular trace]
Let $\nalg$ be an algebra over $\field$ and $\calg$ a central subalgebra. Suppose $\nalg$ is a free module over $\calg$ of rank $n$. Then each $\phi\in\End(_\calg \nalg)$ can be represented by some $\Theta(\phi)\in\mtx{n}{\calg}$ and there is an inclusion $\eta: \nalg\incl\End(_\calg \nalg)$ defined by left multiplication. Let $\tr$ be the trace on $\mtx{n}{C}$ defined by summing up the diagonal elements. Then the \defn{regular trace} $\trreg$ is defined by the composition
\begin{equation}
  %\nalg\overset{\eta}\incl\End(_\calg \nalg)\overset{\Theta}\incl\mtx{n}{\calg}\overset{\tr}\longrightarrow \calg.
  \nalg\xhookrightarrow{\eta}\End(_\calg \nalg)\xhookrightarrow{\Theta}\mtx{n}{\calg}\overset{\tr}\longrightarrow \calg.
  \label{eq:trreg}
\end{equation}
Note that in the case of $\calg=\field$, $\trreg(x)$ is $m$-times the usual trace for $x\in\mtx{m}{\field}$ because $\mtx{m}{\field}$ is a direct sum of $m$ columns as a vector space over $\field$ \cite[2.2-(3)]{BY}. Thus $\Theta$ is proper inclusion for $m>1$ as $n=m^2$ in \eq{eq:trreg}.
\label{ex:reg-tr}
\end{example}

An important class of examples are the Cayley-Hamilton Hopf algebras, which includes group algebras of central extensions of finite groups, quantum universal enveloping algebras, big quantum Borel subalgebras and quantum function algebras.

In $\field[x_1,\cdots,x_n]$, the $k$-th power sum is defined by 
\begin{equation*}
  \psi_k(x_1,\cdots,x_n)=\sum_{i=1}^n x_i^k
\end{equation*}
and for $1\le j \le n$ define
\begin{equation*}
  e_j(x_1,\cdots,x_n)=\sum_{1\le i_1<\cdots<i_j\le n} \Biggl(\prod_{l=1}^j x_{i_l}\Biggl).
\end{equation*}
Then by Newton identities there are $p_i\in\infcy[(i!)^{-1}][x_1,\cdots,x_n]$ for $1\le i \le n$ such that
\begin{equation*}
p_i(\psi_1,\cdots,\psi_i)=e_i
\end{equation*}
as formal polynomials in $\infcy[x_1,\cdots,x_n]$.
\begin{definition}
Let $(\nalg,\calg,\tr)$ be an algebra with trace over $\field$, then the \defn{n-characteristic polynomial} of $r\in\nalg\in\calg[x]$ is 
\begin{equation*}
  \chi_{n,r}(x)\coloneqq x^n+\sum_1^n (-1)^{i} p_i(\tr(r),\cdots,\tr(r^{i}))x^{n-i}.
  %\label{eq:n_char}
\end{equation*}
An algebra with trace $(\nalg,\calg,\tr)$ is a \defn{Cayley-Hamilton algebra of degree d} if it satisfies
\begin{enumerate}
  \item $\tr(1)=d$,
  \item $\chi_{d,r}(r)=0$, $\forall r\in\nalg$.
\end{enumerate}
It is called a \defn{Cayley-Hamilton Hopf algebra of degree d} if furthermore $\halg$ is a Hopf algebra and $\calg$ is a Hopf subalgebra.
\label{def:CH_alg}
\end{definition}
A finitely generated Cayley-Hamilton algebra $(\nalg,\calg,\tr)$ is a finite module over $\tr(\nalg)$, so in particular it is a finite module over $\calg$ \cite[Theorem 2.6]{DPRR}. If $\chark\field=0$, it has an injective represention compatible with trace $\phi:\nalg\to\mtx{d}{B}$ for some commutative algebra $B$, i.e. $\phi$ is an algebra homomorphism s.t. $\tr_{\mtx{d}{B}}\circ\phi=\phi\circ\tr_R$ if and only if $\nalg$ is a Cayley-Hamilton algebra of degree $d$ \cite[Theorem 0.3]{P1}. Here $\tr_{\mtx{d}{B}}$ is computed through taking the sum of diagonal elements. If  $(\nalg,\calg,\tr)$ satisfies \cdk{d}, $\maxi\in\maxspec{\calg}$ and $J$ is the Jacobson radical, then
\begin{equation*}
  (\nalg/\maxi\nalg)/J(\nalg/\maxi\nalg)\cong \oplus_{i=1}^k \mtx{n_i}{\field}
\end{equation*}
for some positive integers $k$, $n_i$ and there are positive integer $s_i$ such that \cite[Proposition 4.3]{DP}
\begin{equation*}
  \sum_{i=1}^k s_i n_i=d.
  %\label{eq:Cayley_dim}
\end{equation*}
\begin{example}
Let  $\nalg$ be a finitely generated algebra over a field $\field$ with $\chark\field\notin[1,r]$ and a free module over a central subalgebra $\calg$ of rank $r$. Then $(\nalg,\calg,\trreg)$ is a Cayley-Hamilton algebra of degree $r$.
\end{example}

\begin{definition}\hfill
\begin{enumerate}
\item  The \emph{$n$-th discriminant ideal} $\Dn{n}$ is defined as the ideal of $C$ generated by elements of the form
\begin{equation*}
 \det(\tr(y_i y_j))_{1\le i,j \le n}, \quad (y_1,\cdots,y_n) \in R^n \;
\end{equation*}
\item and the \emph{$n$-th modified discriminant ideal} $\MDn{n}$ is defined as the ideal of $C$ generated by the elements
\begin{equation*}
 \det(\tr(y_i y'_j))_{1\le i,j \le n}, \quad (y_1,\cdots,y_n), (y'_1,\cdots,y'_n) \in R^n.
\end{equation*}
\item When $\nalg$ is furthermore a free (left) $C$-module of rank $N$, the \emph{discriminant} $\D$ is given by
\begin{equation*}
  \D=\det(\tr(a_i a_j))_{1\le i,j \le N}
\end{equation*}
for a basis $\displaystyle\{a_1,\cdots,a_N\}$ of R as a C-module. \end{enumerate}
\label{def:dist}
\end{definition}
The discriminant is unique up to a factor of the square of a unit of $\calg$ computed from the determinant of the change of basis matrix and
\begin{equation*}
 \Dn{n^2}=\MDn{n^2}=\generate{\D}.
 %\label{eq:D_Dn}
\end{equation*}
Following the literature, $\nalg$ is assumed to be a finite module over $\calg$ throughout this work, so $\nalg$ is a PI ring. If $\nalg$ is furthermore prime and $n$ is its PI-degree, then $\Dn{k}=\MDn{k}=0$ for $k>n^2$ \cite[Corollay 2.4]{BY}. Since the determinant of a matrix can be computed using $\calg$-linear combination of determinants of smaller matrices 
\begin{equation*}
  \calg=\MDn{1}\supseteq\cdots\supseteq\MDn{k}\supseteq\MDn{k+1}\supseteq\cdots.
\end{equation*}
\begin{definition}
If $\algtr{\nalg}$ satisfies \cdk{d} for some $d$, define $\zeros_k$ as in \eq{eq:sd}; then
\begin{equation*}
  \varnothing=\zeros_{1}\subseteq\cdots\subseteq\zeros_{k}\subseteq\zeros_{k+1}\subseteq\cdots.
\end{equation*}
Recall \eq{eq:sd} and $\nalg$ is a finite module over $\calg$, so there is a smallest positive integer $h$ such that
\begin{equation*}
  \varnothing=\zeros_{1}\subseteq\cdots\subseteq\zeros_{h}\subsetneq\zeros_{h+1}=\maxspec{\calg}.
\end{equation*}
Then $\Dn{h}$ is called the \defn{the highest discriminant ideal}.
%\label{def:h_dist}
Similarly, there is a smallest $l$ such that
\begin{equation*}
  \varnothing=\zeros_{1}=\cdots\subsetneq\zeros_{l}\subseteq\cdots.
\end{equation*}
Then $\Dn{l}$ is called the \defn{the lowest discriminant ideal}.
\label{def:hl_dist}
\end{definition}

In a Hopf algebra $\halg$ over a field $\field$, the set of algebra homomorphisms from $H$ to $\field$ denoted $\algkk{\halg}$ is the same as the set of group-like elements in the finite dual of $H$ denoted by $\gpfdh{\halg}$ \cite[(1.3.5)]{Mont}.
%To make the most of the Hopf structure, we consider  pairs of Hopf algebras $(\halg,\calg)$ satisfying the following condition:
%
%\medskip
%\stcond $\halg$ is a finitely generated Hopf algebra over an algebraically closed field $\field$ and a finite module over a central Hopf subalgebra $C$.% and there is a trace $\tr:\nalg\mapsto\calg$.
%\medskip

Suppose $(\halg,\calg)$ satisfies \stcond, then $\calg$ is a finitely generated commutative Hopf algebra and hence the coordinate ring of an affine algebraic group, i.e. there is a multiplication defined on $\maxspec{\calg}$. Denote the subset of $\gpfdh{\halg}$ that vanishes on $\idc$ by $G_0$. There is an action of $G_0$ on $\irr(\halg/\maxi\halg)$ defined by \cite[Section 3.1]{MWY2023}
\begin{equation}
  (\chi,V)\mapsto \chi\otimes_{\field} V\quad\text{for}\;\chi\in G_0, V\in\irr(\halg/\maxi\halg).
  \label{eq:G_0-action}
\end{equation}
%There is a cyclic subgroup of $G_0$ with order equal to $\intord(\halg)$ \cite[Theorem B]{MWY2023}, where $intord(\halg)$ is the integral order of $\halg$ defined in \cite{LWZ}. 
\begin{definition}
Define the \defn{stabilizer} of $V\in\irr(\halg/\maxi\halg)$ under the this action of $G_0$ as
  \begin{equation*}
  \stab{V}\coloneqq \{\chi\in G_0:\chi\otimes V\cong V\}
  \end{equation*}
It can be shown that \cite[Proposition 3.5(b)]{MWY2023}
\begin{equation}
  \ordergp{\stab{V}}{}\le\dim(V)^2
  \label{eq:stab-le}
\end{equation}
The irreducible module $V$ is \defn{maximally stable} if the equality holds.
\label{def:max-stab}
\end{definition}
If $\halg/\idc\halg$ is basic, the group $G_0$ acts transitively on $\irr(\halg/\maxi\halg)$ \cite[Theorem 3.1(c)]{MWY2023}, hence all irreducible modules over $\maxi$ have the same dimension and we get
\begin{proposition}[Theorem B, 4.2~\cite{MWY2023}]
Let $(H,C,\tr)$ be a finitely generated Cayley-Hamilton Hopf algebra of degree $d$ over an algebraically closed field $\field$ of $\chark\field\notin[1,d]$. Suppose the identity fiber algebra $\halg/\idc \halg$ is basic.  Let $\maxi\in\maxspec{\calg}$ and $V\in\irr(\halg/\maxi\halg)$. Then
\begin{equation}
  \sd{\maxi}=\frac{\abs{G_0}\dim(V)^2}{\abs{\stab{V}}}.
\end{equation}
And the following are equivalent
\begin{enumerate}
  \item the maximal ideal $\maxi$ is in the zero set of the lowest discriminant ideal $\zeros_l$;
  \item there is a maximally stable $V\in\irr(\halg/\maxi\halg)$;
  \item any $V\in\irr(\halg/\maxi\halg)$ is maximally stable.
\end{enumerate}
\label{prop:mstable_l}
\end{proposition}

\begin{definition}
Let $\phi\in\gpfdh{\halg}$, the \defn{left winding} ($\lw$) and \defn{right winding} ($\rw$) \defn{automorphisms} in $\autk(\halg)$ are defined by 
\begin{equation*}
  \lw(\phi)(h)=\sum \phi(h_1)h_2,\quad \rw(\phi)(h)=\sum \phi(h_2)h_1,\quad \forall h\in\halg.
\end{equation*}

\label{defn:winding}
\end{definition}
One of the features of winding automorphisms is that both of the automorphism groups $\lw(\gpfdh{\halg})$ and $\rw(\gpfdh{\halg})$ act transitively on kernels of elements in $\gpfdh{\halg}$ or equivalently the set of one-dimensional representations of $\halg$. This follows from the fact the $\gpfdh{\halg}$ is a group under convolution and
\begin{equation*}
  \psi\circ\lw(\phi)=\phi*\psi,\quad \psi\circ\rw(\phi)=\psi*\phi,\quad \forall \phi,\,\psi\in\gpfdh{\halg}.
\end{equation*}
Winding automorphisms are very important in the discussions about homological integrals \cite{LWZ} and dualising complexes \cite{BZ}.

\begin{definition}
Let $\halg$ be a Hopf algebra over a field $\field$ and $\field$ denote the one-dimensional bimodule with actions given by the counit. Then $\halg$ is \defn{ left AS-Gorenstein} if
\begin{enumerate}
  \item the left module $_{\halg} \halg$ has injective dimension $d$;
  \item and\begin{equation}\dim_{\field} \Ext_{\halg}^i(_{\halg} k,_{\halg} \halg)=\begin{cases} 1 \quad i=d,\\0\quad i\ne d.\end{cases}
        \end{equation}
\end{enumerate}
Similarly we can define \defn{right AS-Gorenstein} for the right $H$-module structure. A Hopf algebra $H$ is \defn{AS-Gorenstein} if it is both left and right AS-Gorenstein.
\end{definition}
For Noetherian PI Hopf algebras, these two conditions are equivalent with the same $d$ \cite[Theorem 3.2]{WZ2002}. If a pair $(\halg,\calg)$ satisfies \stcond, then $\halg$ is AS-Gorenstein \cite[Theorem 3.5]{WZ2002}.

\begin{definition}
Let $\halg$ be an AS-Gorenstein Hopf algebra over the field $\field$, then the \defn{left homological integral} is the one-dimensional $\halg$-bimodule 
\begin{equation*}
  \int^l=\Ext^d_{\halg}(_{\halg}\field,_{\halg} \halg)
\end{equation*}
and similary the \defn{right homological integral} is defined by 
\begin{equation*}
  \int^r=\Ext^d_{\halg^{\opp}}(\field_{\halg}, \halg_{\halg})
\end{equation*}
where $\halg^{\opp}$ is the opposite ring.
\end{definition}
\begin{definition}
The \defn{integral order} of an AS-Gorenstein Hopf algebra is the smallest positive integer $n$ making $(\int^r)^{\otimes n}\cong \field$ as $\halg$-bimodules and $\infty$ if not such $n$ exists.
\end{definition}
\begin{remark}
Let $(H,C)$ satisfy \stcond, then the antipode $S$ of $\halg$ is bijective \cite{Sk}, thus $S(\int^r)=\int^l$ and $S(\int^l)=\int^r$. Consequently, $\intord(\halg)=1 \Leftrightarrow \int^r\cong\field \Leftrightarrow \int^l\cong\field$.
\end{remark}

\section{The Second Cohomology Group and Fiber Algebras}
%The first part of this section is a short review of classifying central extensions of groups using the second cohomology group. We refer the reader to \cite{Men} and its reference for more detailed background on the topic. In the second part we show that for a specific class of central extensions that we consider, the associated second cohomology group can be represented by some matrices.
In this section, we use the second group cohomology to understand the fiber algebras of group algebras of central extensions of finite groups and then prove \thm{tthm:dist} and \coro{tcor:basic}. A friendly introduction to the second cohomology group can be found in \cite{Men} and more detailed expositions are availabel in its references including \cite{KAR}.
\begin{definition}
Let $G$ be a group, and $\cexta$ be an abelian group in additive notation. A function $\tccl:G\times G\to \cexta$ is called a \emph{2-cocycle} if
\begin{equation*}
  \tccl(h,k)-\tccl(gh,k)+\tccl(g,hk)-\tccl(g,h)=0\quad\forall g,h,k\in G
  %\label{eq:2_cocycle}
\end{equation*}
and a \emph{$2$-coboundary} if there a function $f:G\to \cexta$ such that
\begin{equation*}
  \tccl(g,h)=f(h)-f(gh)+f(g)\quad\forall g,h\in G.
%  \label{eq:2_coboundary}
\end{equation*}
\end{definition}
Denote the set of 2-cocyles and 2-coboundaries by $\cocyl{G}{\cexta}$ and $\cobnd{G}{\cexta}$, respectively. These have structures of abelian groups from $\cexta$. And the \emph{2nd-cohomology group} is defined as the quotient group
\begin{equation*}
  \cohom{G}{\cexta}=\frac{\cocyl{G}{\cexta}}{\cobnd{G}{\cexta}}.
%  \label{eq:2_cohomology}
\end{equation*}
Two 2-cocyles are called \emph{cohomologous} if they are in the same cohomology class $\cohom{G}{\cexta}$. Every 2-cocycle satisfies% \cite[Lemma 1.2.1]{Men}
\begin{align}
  \tccl(1,g)&=\tccl(1,1)=\tccl(g,1) \quad \text{and}\\
  \tccl(g,g^{-1})&=\tccl(g^{-1},g)\quad\forall g\in G
  \label{eq:normal_2}
\end{align}
and is cohomologous to a normalized one satisfying $\tccl(1,1)=1$.

%Every 2-cocycle is cohomologous to one that satisfies \cite[Lemma 1.2.1]{Men}
%\begin{align}
%  \tccl(1,g)&=\tccl(1,1)=\tccl(g,1)=1 \quad \text{and}\\
%  \tccl(g,g^{-1})&=\tccl(g^{-1},g).\quad\forall g\in G.
%  \label{eq:normal_2}
%\end{align}

%Let $\egp$ be a central extension of a group $G$ by another abelian group $\agp$, then up to equivalence of extension, $\egp$ is isomorphic to by a 2-cocycle $\displaystyle\tccl \in \ntccl{\bgp}{\agp}$ satisfying \eq{eq:normal_2} given by
%Let $\egp$ be a central extension of a group $G$ by another abelian group $\agp$, then up to equivalence of extension, $\egp$ is isomorphic to the group $\agp\rtimes_{\tccl}G\coloneqq$ defined by
%Let $G$ be a group and $\agp$ be an abelian group in additive notation, define
%\begin{align}
%   \agp\rtimes_{\tccl}G &\coloneqq\{(c,a):c\in \agp, \,a\in G\},\\
%  (c_1,g_1)(c_2,g_2) &=(c_1+c_2+\tccl(g_1,g_2),g_1 g_2), \quad c_i\in\agp,g_i\in G.
%\end{align}
%Then any central extension of $G$ by $M$ is equivalent to $\agp\rtimes_{\tccl}G$ for some 2-cocyle $\tccl$ and two such extensions are equivalent if and only if the cocyles are cohomologous \cite[Theorem 3.2.3]{Men}.
%In additive notation of abelian groups, a representative of $\tccl$ can be chosen to be a normalized version such that

%We consider the case when $G=\bgp=A\times B$ and $\tccl$ and $(A,0)$, $(0,B)$ are abelian in $\egp$ or equivalently
%We consider normalized 2-cocycles $\tccl \in \ntccl{\bgp}{\agp}$ that also satisfy the additional condition
\begin{definition}
Let $\cextb$ be a group, $\cexta$ be an abelian group, then a \defn{central extenstion of $\cextb$ by $\cexta$} is a short exact sequence of groups
\begin{equation}
  1\longrightarrow \cexta\overset{i}{\longrightarrow} \cexte\overset{p}{\longrightarrow} \cextb\longrightarrow 1
  \label{eq:gr_cext}
\end{equation}
such that the image of $i$ is in the center of $\cexte$.
\end{definition}
\begin{proposition}
Let $\halg$ be a finitely generated cocommutative Hopf algebra over an algebraically closed field of characteristic zero and module finite over be a central Hopf subalgebra $\calg$. Then there is a central extension $\cexte$ of a finite group $\cextb$ by a finitely generated abelian group $\cexta$ and a nonnegative integer $n$ such that $\halg=\field \cexte\otimes \field[x_1,\cdots,x_n]$ and $\calg=\field \cexta\otimes \field[x_1,\cdots,x_n]$.
\label{prop:general}
\end{proposition}
\begin{proof}
By Artin-Tate Lemma, $\calg$ is a finitely generated $\field$-algebra and thus Noetherian by Hilbert's Basis Theorem and; $\halg$ is also Noetherian as a finite module over $\calg$. By Cartier-Gabriel-Kostant Theorem, $\halg\cong \field \cexte \ltimes U(\g)$ for some group $\cexte$ and finite dimensional Lie algebra $\g$. Now $U(\g)$ is infinite-dimensional and there is no central Hopf subalgebra of $U(\g)$ over which $U(\g)$ is a finite module \cite[5.1.6]{Mont}\cite[Lemma 7.2]{Zhuang}, so $\g$ is abelian and $U(\g)$ is in the center of $\halg$. This shows $\halg\cong\field\cexte\otimes\field[x_1,\cdots,x_n]$ for some nonnegative integer $n$. Now $\field \cexte$ is Noetherian and a finite module over a central Hopf subalgebra. Applying Artin-Tate Lemma and Cartier-Gabriel-Kostant Theorem again, we prove the result.
\end{proof}
Let $\tccl\in\cocyl{\cextb}{\cexta}$
and $G_{\tccl}$ be a set with multiplication given below
%\begin{remark}
%It turns out that in characteristic zero, if $H=\field G$ and $(H,C)$ satisfy \stcond, the group $G$ can be realized as the central extension of a finite group by a finitely generated abelian group: By Cartier-Gabriel-Kostant Theorem, $C\cong \field A \ltimes U(\g)$ for some group $A$ and Lie algebra $\g$. But there are no primitive elements in the Hopf algebra $H=\field G$, so $C=\field A$. By Artin-Tate Lemma, $C$ is a finitely generated commutative $\field$-algebra and is Noetherian. As a finite module over $C$, the algebra $H$ is also Noetherian. Thus both $G$ and $A$ are finitely generated and $G$ is a central extension of $G/A$ by $A$. 
%\label{rm:A-is-group}
%\end{remark}
%
%Let $G$ be a group, $M$ be an abelian group and $\tccl\in\cocyl{G}{M}$, denote by $G_\tccl$ the central extension  of $G$ by $M$ with 
\begin{align}
   G_{\tccl} &\coloneqq\{(m,g):m\in \cexta, \,g\in \cextb\},\label{eq:g_ext_a}\\
  (m_1,g_1)(m_2,g_2) &\coloneqq(m_1+m_2+\tccl(g_1,g_2),g_1 g_2), \quad m_i\in \cexta,g_i\in \cextb.\label{eq:g_ext_b}
\end{align}
Then $G_{\tccl}$ is a central extension of $\cextb$ by $\cexta$ and two such extensions are equivalent if and only if the associated 2-cocyles are cohomologous.
We can associate a $\tccl\in\cocyl{\cextb}{\cexta}$ with the central extension in \eq{eq:gr_cext} as follows: fix $f$ that is a section of $p$ as a set map and define $\tccl:\cextb\times \cextb\mapsto \cexta$ by $\tccl(g_1,g_2)=i^{-1}(f(g_1)f(g_2)f(g_1 g_2)^{-1})$ for all $g_1,g_2\in \cextb$. Then $G_{\tccl}$ is a central extension of $\cextb$ by $\cexta$ equivalent to $\cexte$.

\begin{example}
Consider the central extension of $\Zn{2}$ by $\bbz$ given by
\begin{equation*}
  1\longrightarrow \bbz\overset{\times 2}{\longrightarrow} \bbz\overset{p}{\longrightarrow} \Zn{2}\longrightarrow 1
\end{equation*}
The function $f:\Zn{2}\mapsto \bbz$ defined by $f(0)=0$ and $f(1)=1$
%\begin{equation*}
%  f(0)=0,\;\text{and}\;f(1)=1
%\end{equation*}
is a section of $p$ and induces $\tccl\in\cocyl{\Zn{2}}{\bbz}$ given by
\begin{equation*}
  \tccl(i,j)=\begin{cases} 1\quad\text{if}\;i+j=2\\0\quad\text{otherwise}. \end{cases}
\end{equation*}
The group $G_{\tccl}$ is then an infinite cyclic group isomorphic to $\bbz$ generated by $(0,1)$ as follows:
\begin{equation*}
  n(0,1)=(n/ 2,n\pmod2)\quad \forall n\in \bbz.
\end{equation*}
\end{example}

%Any central extension of $G$ by $M$ is equivalent to $G_{\tccl}$ for some $\tccl\in\cocyl{G}{M}$ and two such extensions are equivalent if and only if the cocyles are cohomologoues \cite[Theorem 3.2.3]{Men}.

\begin{definition}

Let $G$ be a group and $\tccl\in\cocyl{G}{\fieldx}$. The \defn{twisted group algebra of $G$ by $\tccl$} denoted $\field_{\tccl} G$ is a vector space over $\field$ with basis $\{g\}_{g\in G}$ and multiplication defined by
\begin{equation}
  g_1 \cdot g_2=\bar{\tccl}(g_1,g_2)g_1 g_2\quad\forall g_1,g_2\in G.
  \label{eq:twisted_gp_alg}
\end{equation}
The projective representations of $G$ correspond to modules over twisted groups algebras of $G$. Two twisted group algebras $\field_\beta G$ and $\field_\beta$ are \defn{equivalent} if there is an isomorphism of algebras $\phi:\field_\beta G\mapsto\field_\gamma G$ and there is a map $t:G\mapsto \fieldx$ s.t. $\phi(g)=t(g)g$ for any $g\in G$. And this happens if and only if the 2-cocyles are cohomologous.
\end{definition}
\begin{remark}
A function $\tccl:\cextb\times \cextb\mapsto \cexta$ is a 2-cocyle if and only if \eq{eq:g_ext_b} defines an associative multiplication. Similarly associativity of the operation defined in \eq{eq:twisted_gp_alg}  is equivalent to the  map $\tccl:\cextb\times \cextb\mapsto \field$ being a 2-cocycle is $\cocyl{\cextb}{\fieldx}$.% Thus if $\abs{\cextb}<\infty$, the a 2-cocyle $\tccl{\cextb}{\fieldx}$ can be recorded in some $\abs{\cextb}\times \abs{\cextb}$ matrix over $\field$.
\label{rm:2co-twgp}
\end{remark}
\begin{proof}[Proof of \thm{tthm:dist}]
\hfill
\par
\ref{tthm:dist:CH} The algebra $\halg$ is a free $\calg$-module of rank $\abs{\cextb}$ as the group $\cexte$ has $[\cexte:\cexta]=\abs{\cextb}$ left or right cosets of $\cextb$. This part is now clear from \eg{ex:reg-tr}.

\ref{tthm:dist:di} Let $\maxi\in \maxspec \calg$, since $\calg$ is affine and $\field$ is algebraically closed, $\calg/\maxi\cong\field$. The 2-cocyle $\tccl\in\cocyl{\cextb}{\cexta}$ associated with the central extension descends to a 2-cocyle $\bar{\tccl}\in\cocyl{\cextb}{\fieldx}$ by \rmk{rm:2co-twgp}. This means $\halg/\maxi\halg$ is isomorphic to the twisted group algebra $\field_{\bar{\tccl}} \cextb/A$. Now the twisted group algebra of a finite group $\cextb$ by a 2-cocyle $\tccl\in\cocyl{\cextb}{\fieldx}$ is semsimple whenever $\chark\field$ does not divide $\abs{\cextb}$ \cite[Theorem 2.2]{Pass}. Thus for any $\maxi\in\maxspec{\calg}$, the fiber algebra $\halg/\maxi \halg$ is always semisimple and 
\begin{equation*}
\sd{\maxi}=\dim_{\field}\halg/\maxi\halg=[G:A]
\end{equation*}
The zero sets of the discriminant ideals can now be computed using \eq{eq:sd}. 

\ref{tthm:dist:io} Let $c$ be one of the standard generators of the torsion-free part of the abelian group $i(\cexta)$ in the center of $\cexte$ and $I=\generate{c-1}$. Then $\halg/I\cong \field (G/\generate{c})$ is a Hopf algebra over $\field$. Note $c-1$ is a nonzero divisor of $\halg$ since $\halg$ is $\bbz$-graded as a $\field[c,c^{-1}]$-module. Hence $\intord(\halg)=\intord(\halg/I)$ \cite[Lemma 2.6]{LWZ}. Repeating this process, we end up with $\intord(\halg)=\intord(\field \bar{G})$ for a finite group $\bar{G}$ and $\intord(\halg)=\intord(\field \bar{G})=1$ since $\field\bar{G}$ is semisimple by Maschke's Theorem.
\end{proof}

\begin{proof}[Proof of \coro{tcor:basic}]
\hfill
\par
\ref{tcor:basic:semi} 
 The result follows from the transitivity of the action of the group $G_0$ on $\irr(\halg/\maxi\halg)$ and \thm{tthm:dist}\ref{tthm:dist:di}.

\ref{tcor:basic:stab} This part is a consequence of \thm{tthm:dist}\ref{tthm:dist:di} and \prop{prop:mstable_l}.

\ref{tcor:basic:orbit}
% By \thm{tthm:dist}, all fiber algebras are semisimple, so a fiber algebra is basic if and only if it is abelian.
Associate a $\tccl\in\cocyl{\cextb}{\cexta}$ with the central extension of $\cextb$ by $\cexta$ and define $A$ as the subgroup of $\cexta$ generated by elements in the set 
\begin{equation*}
  \{\tccl(g_1,g_2)-\tccl(g_2,g_1)\},\quad g_1,g_2\in\cextb.
\end{equation*}
Let $\maxi\in\maxspec{\calg}$, then $\halg/\maxi\halg$ is semisimple by \thm{tthm:dist}\ref{tthm:dist:di}. Hence it is basic if and only if it is commutative or equivalently every element in the group $A$ is sent to the number $1$ by the natural projective $\halg\to\halg/\maxi\halg$.
\end{proof}

\section{Central Extensions of Finitely Generated abelian Groups}
\label{sec:central-ext}
In this section we give another proof of \thm{tthm:dist}\ref{tthm:dist:di} for group algebras of central extensions of abelian groups and compute explicitly the orbits in $\maxspec{\calg}$ under different automorphism groups.
%In this section we consider group algebras of central extensions of finitely generated abelian groups, i.e. there is a short exact sequence of groups

Let
\begin{equation*}
  1\longrightarrow \agp\longrightarrow \egp\longrightarrow \bgp\longrightarrow 1
\end{equation*}
be a central extension of a finitely generated abelian group  $\bgp$ by a finite abelian group $\agp$. Recall $\egp \cong G_{\tccl}$ for some $\tccl\in\cocyl{\bgp}{\agp}$. Let $\{x_i\}_{1\le i\le n}$ be the standard generators of $\bgp$ in the invariant factor form and define $a_{ij}=\tccl(x_j,x_i)-\tccl(x_i,x_j)$ then
\begin{align*}
  a_{ij}&=a_{ji}^{-1} \quad\text{and}\\
  x_i \cdot x_j&=a_{ij} x_j \cdot x_i.
\end{align*}
Denote by  $\ordergp{\cdot}{\agp}$ the order of an element in $\agp$, set
\begin{equation*}
  l_i=\underset{1\le j \le n}{\lcm} \ordergp{a_{ij}}{\agp}\quad\forall 1\le i\le n.
\end{equation*}
The set $\{\Im\tccl,x_1^{l_1},\cdots, x_n^{l_n}\}$ is in the center $Z(\egp)$. Pick a subgroup $\cgp\lhd Z(\egp)$ containing this set, then $\abs{\egp:\cgp}<\infty$ and $G_0\cong\egp/\cgp$. In particular there is also a central extension
%Conversely, any central subgroup of $\exte$ of finite index contains such as set. 
\begin{equation*}
  1\longrightarrow \cgp\longrightarrow G_{\tccl}\longrightarrow G_0\longrightarrow 1.
\end{equation*}
Define $\halg=\field\egp$ and $\calg=\field \cgp$; the pair $(H,C)$ satisfies \stcond. By \thm{tthm:dist}\ref{tthm:dist:di}, the zero set of the $k$-th discriminant ideal is empty if $k\le\abs{\egp:\cgp}$ and the whole $\maxspec C$ otherwise.

The result can also be deduced from \prop{prop:mstable_l}. By \thm{tthm:dist}\ref{tthm:dist:di} $\halg/\maxi\halg\cong\field_{\gamma} G_0$ for a 2-cocycle $\gamma\in\cocyl{G_0}{\fieldx}$. The images $\bar{x}_i$ of the generators $x_i$ under the natural projection $\egp\mapsto \egp/\cgp$ can be extended to a set of standard generators (possibly with $0's$) of $\egp/\cgp\cong G_0$ in the invariant factor form $\{\bar{x}_1,\cdots,\bar{x}_{n+s}\}$. Note in $\halg/\maxi\halg$ elements in $\cgp$ correspond to elements in $\fieldx$. Thus the 2-cocyle $\gamma$ can be defined as:
\begin{equation*}
 \gamma(\bar{x}_i,\bar{x}_j)=\begin{cases}
                               \tccl(x_i,x_j)+\maxi\in\fieldx\quad \forall 1\le i,j\le n,\\
			       1\quad \text{otherwise}.
                             \end{cases}
\end{equation*}
As the group $G_0$ is a finite abelian group, each $a_{ij}$ descends to a root of unity $\bar{a}_{ij}$. There is a root of unity $\runit$ of some order $\orderu$ that generates all $\bar{a}_{ij}$ and the multiplication on $\field_{\gamma}G_0$ can be represented by a matrix $\cmmtx(\xi)\in\mtx{n+s}{\Zn{\orderu}}$ dependent on the choice of root of unity with entries defined by $\xi^{\cmmtxij_{ij}}=\bar{a}_{ij}$. The matrix $\cmmtx(\xi)$ is skew-symmetric; thus, by a change of $\bbz$-basis of $\cextb$ or switching to a new set of generators $\{y_1,\cdots,y_{n+s}\}=\{\phi(\bar{x}_1),\cdots,\phi(\bar{x}_{n+s})\}$ for some $\phi\in\autk(\halg/\maxi\halg)$, it can be written in the skew normal form \cite[Theorem IV.1]{MN1972}

%The result can also be deduced from \prop{prop:mstable_l}. As shown in the last section, $\halg/\maxi\halg\cong\field_{\bar{\tccl}} G$ for the 2-cocycle $\bar{\tccl}\in\cocyl{\bar{G}}{\fieldx}$ induced from $\tccl$. As a quotient of $\bgp$, the group $\egp/\cgp\cong G_0$ is a finite abelian group, thus each $a_{ij}$ descends to a root of unity $\bar{a}_{ij}$. There is a $m$-th root of unity $\runit$ such that for any $i, j$ there is some integer $0\le n_{ij}<m$ s.t. $\bar{a}_{ij}=\runit^{n_{ij}}$. The matrix $N\in\mtx{n}{\bbz}$ with entries defined by $n_{ij}$ is skew-symmetric ($\mod m$, thus by a change of $\bbz$-basis or in some set of minimal generators $\{y_1,\cdots,y_n\}$ of $\bar{G}$, the matrix $N$ can be written in the skew normal form \cite[Theorem IV.1]{MN1972}
\begin{equation}
\cmmtx(\xi)=\begin{bmatrix}
   0 & k_1 & & & & & &\\
   -k_1 & 0 & & & & & &\\
   & & \ddots & & & & &\\
   & & & 0 & k_s & & &\\
   & & & -k_s & 0 & & &\\
   & & & & & 0 & &\\
   & & & & & & \ddots &\\
   & & & & & & & 0\\
\end{bmatrix}.
\label{eq:N}
\end{equation}
%where each $k_i$ is an integer not divisible by $m$. As they are powers of $\runit$ we may assume $k_i |m$ for any $1\le i\le s$.

The first step is to find one irreducible representation of $\halg/\maxi\halg$. Denote by $l_i$ the order of $k_i$ in $\Zn{\orderu}$. Let $I$ be the ideal of $\halg/\maxi\halg$ generated by $(y_{2s-1})^{l_i}-1$ and $\{y_{2s}^{l_i}-1\}$ for $1\le k\le s$ and $y_i-1$ for $i>2s$. Then
\begin{align*}
  (\halg/\maxi\halg)/I&= R_1\tensork \cdots\tensork R_s\quad \text{where}\\
  R_i&\cong\frac{\field\generate{x,y}}{(x^{l_i}-1,y^{l_i}-1,xy-\runit^{k_i}yx)}\quad\forall 1\le i\le s.
\end{align*}
We recall the following result in the simpliest case:
\begin{lemma}[Section 5.1 \cite{MWY2023}]
  Let $\tilde{\bgp}=\Zn{n}\times\Zn{n}$, $\tilde{\agp}=\Zn{n}$ and $\tilde{\egp}$ be a central extension of $\tilde{\bgp}$ by $\tilde{\agp}$ with associated $\tilde{\tccl}\in\cocyl{\tilde{\bgp}}{\tilde{\agp}}$ as follows
  \begin{equation*}
    \tilde{\tccl}(a,b)=c,\;\tilde{\tccl}(b,a)=0
  \end{equation*}
where $a=(1,0),b=(0,1)\in\cextb$ and $c=1\in\cexta$ in additive notation. Define $\tilde{\halg}=\field \tilde{\egp}$ and $\tilde{\calg}=\field \tilde{\agp}$. Let $\tilde{\maxi}=(c-\runit)\in\maxspec{\tilde{\calg}}$ where $\runit$ is an $n$-th primitive root of unity. Then 
%We identify $a$, $b$ and $c$ with the corresponding generators of $\tilde{\egp}$ and $\tilde{\agp}$ with a central subgroup of $\tilde{\egp}$. 
  \begin{enumerate}
    \item $\tilde{\halg}/\tilde{\maxi}\tilde{\halg}$ has an irreducible representation $V$ of dimension $n$; thus, it is simple;
    \item The stabilizer of $V$ is given by
    \begin{align*}
       \stab(V)=&\{\chi\in\gpfdh{\tilde{\halg}}:\; \chi(a)=\runit^i,\,\chi(b)=\runit^j, \, i,j\in\infcy\;\mathrm{mod}\; n\}\\
       \cong& \Zn{n}\times\Zn{n}.\nonumber
       \end{align*}
  \end{enumerate}
  \begin{proof}
    (i) The fiber algebra
    \begin{equation*}
      \tilde{\halg}/\tilde{\maxi}\tilde{\halg}\cong\frac{\field<x,y>}{(x^n-1,y^n-1,xy-\runit yx)}.
    \end{equation*}
    Let $V$ be a vector space with basis $\displaystyle \{v_0,\cdots,v_{n-1}\}$. Define
    \begin{equation*}
      x \cdot v_i=\runit^i v_i, \quad y \cdot v_i=v_j,\quad j=i+1 \;\text{mod}\; n.
    \end{equation*}
    Then $V$ is an n-dimensional representation of $\tilde{\halg}/\tilde{\maxi}\tilde{\halg}$.

    (ii) \cite[Section 5.1]{MWY2023}.
  \end{proof}
  \label{lm:dim_onepair}
\end{lemma}
By \lm{lm:dim_onepair}, the quotient $(\halg/\maxi\halg)/I$ is simple as a tensor product over $\field$ of simple algebras whose centers are all $\field$ \cite[Theorem 1.7.27]{RW1}; thus, $\halg/\maxi\halg$ has an irreducible representation $V$ of dimension $\prod_{1\le i\le s} l_i$ and
\begin{equation*}
  L=(\Zn{l_1})^2\times\cdots\times(\Zn{l_s})^2\subset\stab{V}.
\end{equation*}
As $\bgp$ is abelian, $\halg/\idc \halg$ is basic; by \eq{eq:stab-le} $L=\stab{V}$. Thus from \prop{prop:mstable_l} we recover \thm{tthm:dist}\ref{tthm:dist:di}.

We note that the generators $\bar{x}_i$ can be chosen directly from $\cextb$. The results are summarized below
\begin{proposition}
Let $\egp$ be a central extension of a finitely generated abelian group $\bgp$ by a finite abelian group $\agp$. There exits a central subgroup $\cgp\lhd\egp$ of finite index. Define $\halg=\field \egp$ and $\calg=\field\cgp$. Let $\{x_1,\cdots,x_n\}$ be a set of standard generators of the quotient group $\cextb=\egp/\cgp$ and $\maxi\in\maxspec{\calg}$. Then $\halg/\maxi\halg\cong \field_{\tccl} \cextb$ for some $\tccl\in\cocyl{\cextb}{\fieldx}$. Furthermore, there is $\phi\in\autk(\halg/\maxi\halg)$ and a root of unity $\xi$ of some integer order $\orderu$ such that another set of generators $\{y_1,\cdots,y_{n}\}=\{\phi(\bar{x}_1),\cdots,\phi(\bar{x}_{n})\}$ satisfy $y_i\cdot y_j=\cmmtxij_{ij} y_j \cdot y_i$ for a matrix $\cmmtx\in\mtx{n}{\Zn{\orderu}}$ in the form of \eq{eq:N}. Let $l_i$ be the order of $k_i\in\Zn{\orderu}$, and $V\in\irr(\halg/\maxi\halg)$ then
\begin{align*}
  \stab{V}&\cong(\Zn{l_1})^2\times\cdots\times(\Zn{l_s})^2\quad\text{and}\\
  \dim V&=\prod_{i=1}^s l_i.
\end{align*}
\label{prop:c-ext}
\end{proposition}

\begin{example}
Let $\bgp=\Zn{3}\times\Zn{3}\times\bbz\times\bbz$ and $\agp=(\Zn{3})^2$. Denote the standard generators by $a_1,a_2,a_3,a_4\in \bgp$ and $c_1,c_2\in \agp$. Fix $\tccl\in\cocyl{\bgp}{\agp}$ given by
\begin{equation*}
 \tccl(x,y)=
   \begin{cases}
     (1,0)\quad x=a_1=(1,0,0,0),y=a_2=(0,1,0,0);\\
     (0,1)\quad x=a_1=(1,0,0,0),y=a_4=(0,0,0,1);\\
     (0,0)\quad\text{otherwise}.
   \end{cases}
\end{equation*}
%\nocite{*}
We may identify $\egp$ with $\bgp_{\tccl}$ described in \eq{eq:g_ext_a} and \eq{eq:g_ext_b}. Choose $\cgp=\generate{c_1,c_2,a_3,a_4^3}$ and $\field=\mathbb{C}$. Set $\halg=\mathbb{C}\egp$, $\calg=\mathbb{C}\cgp$. Then $\maxspec{C}\cong (\Zn{3})^2\times\mfield^2$ where $\mfield$ is the multiplicative group of nonzero elements in $\field$, more concretely define $b\coloneqq a_4^3$ and denote by $\runit$ a cubic primitive root of unity; then
\begin{align*}
  \maxspec{C}=\{&\maxi_{u,v,w,x}=\generate{c_1-u,c_2-v,a_3-w,b-x}:\nonumber\\
  &u,v\in\{1,\runit,\runit^2\};w,\,x\in\field\setminus\{0\}\}.
\end{align*}
Characters over $\idc$ are
\begin{align*}
  G_0&=\{\chi\in\gpfdh{\halg}:\chi(c_1)=\chi(c_2)=\chi(a_3)=1;\,\chi(a_1),\chi(a_2),\chi(a_4)\in\{1,\runit,\runit^2\}\}\nonumber\\
  &\cong \egp/\cgp\cong (\Zn{3})^3.
\end{align*}
And \eq{eq:main_dist} becomes
    \begin{equation*}
      \zeros(D_k(\halg/\calg,\trreg))=\begin{cases}
			       \varnothing,\quad &k \le 27,\\
			       \maxspec {C}\cong (\Zn{3})^2\times\mfield^2,\quad &k>27.
                             \end{cases}
    \end{equation*}
 The Hopf algebra $\halg$ is a free $\calg$-module of rank $27$ with basis
 \begin{equation*}
   \{a_1^i a_2^j a_4^k: i,j,k\in\{0,1,2\}\}.
 \end{equation*}
 For any $\maxi_{u,v,w,x}\in\maxspec{C}$
\begin{align*}
  \halg/\maxi_{u,v,w,x}\halg\cong \frac{\field\generate{X^{\pm 1},Y^{\pm 1},Z^{\pm 1}}}{(X Y-u Y X,X Z-v Z X,Z^3-x)}.
\end{align*}
Up to the choice of third primitive root of unity there are five cases determined by $u$ and $v$. Fix an $\maxi\in\maxspec{\calg}$ we can pick a different set of generators $\{y_1,y_2,y_3,y_4\}$ of $\halg/\maxi\halg$ as a $\field$-algebra written in $\bar{a}_i$, i.e. the images under natural projection of $a_i$ to simplify the matrix \eq{eq:N} describing the multiplications. As $G_0\cong\bgp$ is abelian, $\stab{\maxi}$ does not depend on $V\in\irr(\halg/\maxi\halg)$\cite[Proof of Lemma 3.9]{MWY2023}. The five cases are listed in \tab{tab:inv_mtx}.
\begin{table}
\centering
\begin{tblr}{ |c| c| c| c| c| }
\toprule
Case & $\{u,v\}$ & $\{y_1,y_2,y_3,y_4\}$  & $\cmmtx(\xi)$ & $\{\{\chi(y_i)\}:\chi\in\stabb_{G_0}\}$\\
\midrule\\
I & $\{1,1\}$ & $\{\bar{a}_1,\bar{a}_2,\bar{a}_3,\bar{a}_4\}$ & $\begin{bmatrix} 0 & 0 &\\0 & 0 &\\ & & \ddots \end{bmatrix}$ & $\{1,1,1,1\}$\\
\midrule\\
II & $\{\runit,1\}$ & $\{\bar{a}_1,\bar{a}_2,\bar{a}_3,\bar{a}_4\}$ & $\begin{bmatrix}   0 & 1 &\\ -1 & 0 &\\ & &  \ddots  \\\end{bmatrix}$ & $\{\runit^i,\runit^j,1,1\}_{i,j\in\bbz}$\\
\midrule\\
III & $\{1,\runit\}$ & $\{\bar{a}_1,\bar{a_4},\bar{a}_3,\bar{a}_2\}$ &  $\begin{bmatrix}   0 & 1 &\\ -1 & 0 &\\ & &  \ddots \\\end{bmatrix}$ & $\{\runit^i,\runit^j,1,1\}_{i,j\in\bbz}$\\
\midrule\\
IV & $\{\runit,\runit\}$ & $\{\bar{a}_1,\bar{a}_2,\bar{a}_3,\bar{a}_4\bar{a}_2^{-1}\}$ & $\begin{bmatrix}   0 & 1 &\\-1 & 0 &\\ & &  \ddots  \\\end{bmatrix}$ & $\{\runit^i,\runit^j,1,1\}_{i,j\in\bbz}$\\
\midrule\\
V & $\{\runit,\runit^2\}$ & $\{\bar{a}_1,\bar{a}_2,\bar{a}_3,\bar{a}_4\bar{a}_2^{-2}\}$ & $\begin{bmatrix}   0 & 1 &\\-1 & 0 &\\ & &  \ddots  \\\end{bmatrix}$ & $\{\runit^i,\runit^j,1,1\}_{i,j\in\bbz}$\\
\bottomrule
\end{tblr}
\caption{Five cases of fiber algebras}
  \label{tab:inv_mtx}
\end{table}

Pick $\maxi=\maxi_{\runit,\runit^2,w,x}$ in the fifth case, the irreducible modules are three-dimensional. Let $V$ be a three dimensional vector space with basis $\{v_1,v_2,v_3\}$ and define an $\halg$-module structure on $V$ by 
%\begin{align}
%  c_1\cdot v_1&=c_2\cdot v_1=-v_1,\,c_1\cdot v_2=c_2\cdot v_2=-v_2;\\
%  a_1\cdot v_1&=-v_1,\,a_1\cdot v_2=-v_2,\,a_3\cdot v_1=w\cdot v_1,\,a_3\cdot v_2=w\cdot v_2;\\
%  a_2\cdot v_1&=a_4\cdot v_1=v_2,\,a_2\cdot v_2=a_4\cdot v_2=v_1.
%\end{align}
%\begin{align}
%  c_1&=c_2=a_1=\begin{bmatrix} -1 & 0\\0 & -1 \end{bmatrix} \\
%  a_2&=a_4=\begin{bmatrix} 0  & 1\\1 & 0 \end{bmatrix} \\
%  a_3&=\begin{bmatrix} w & 0\\0 & w \end{bmatrix}.
%\end{align}
\begin{align*}
  c_1&=c_2=\begin{bmatrix} \runit & & \\ & \runit & \\ & & \runit\\ \end{bmatrix}, 
  a_1=\begin{bmatrix} 1  & &\\ & \runit & \\&&\runit^2\end{bmatrix},
  a_2=\begin{bmatrix} 0 & 1 &0\\ 0& 0&1\\1 & 0 &0\end{bmatrix}, 
  a_3=\begin{bmatrix} w & & \\ & w & \\ & & w\\ \end{bmatrix},\\ 
  a_4&=\begin{bmatrix} 0 & 0 &\sqrt[3]{x}\\ \sqrt[3]{x}& 0&0\\0 & \sqrt[3]{x} &0\end{bmatrix}; 
\end{align*}
where $\sqrt[3]{x}$ is a cubic root of $x$ then $V$ is an irreducible $\halg$-module over $\maxi$. The stabilzer is given by
\begin{align*}
  \stab{V}=\{\chi\in\gpfdh{\halg}:&\chi(c_1)=\chi(c_2)=\chi(a_3)=1;\nonumber\\
  &\chi(a_1),\chi(a_2)=\chi(a_4)^{2}\in\{1,\runit,\runit^2\}\}.
\end{align*}
The orbit of $\idc$ under winding automorphism groups is the subset of $\maxspec{\calg}$ with basic fibers; thus,
\begin{align*}
    \orbidc &=\autk(\halg,\calg)\cdot \idc=\lw(\gpfdh{\halg})\cdot \idc=\rw(\gpfdh{\halg})\cdot \idc\nonumber\\
    &=\gpfdh{\calg}= \{\maxi_{1,1,w,x}:w,\,x\in\mfield\}\cong\mfield^2.
\end{align*}
By considering the multiplication on the affine algebraic group $\maxspec{\calg}$, we see that for $\maxi\in\maxspec{\calg}$ written as $\maxi_{\xi^i,\xi^j,w,x}$ for some $i,j\in\Zn{3}$ and $w,x\in\mfield$
\begin{align*}
    \orbidc_{i,j} &=\lw(\gpfdh{\halg})\cdot \maxi_{\xi^i,\xi^j,w,x}=\rw(\gpfdh{\halg})\cdot \maxi_{\xi^i,\xi^j,w,x}\nonumber\\
    &= \{\maxi_{\xi^i,\xi^j,w,x}:w,\,x\in\mfield\}\cong\mfield^2.
\end{align*}
Comparison with \tab{tab:inv_mtx} shows that the action of winding automorphism groups is far from transitive on subsets of $\maxspec{\calg}$ with nontrivial isomorphic stablizers. Note in general $\stab{\maxi_1}\cong\stab{\maxi_2}$ is a stronger condition than $\dim V_1=\dim V_2$ for some $V_1\in\irr(\halg/\maxi_1\halg)$, $V_2\in\irr(\halg/\maxi_2\halg)$ or equivalently $\halg/\maxi_1\halg\cong\halg/\maxi_2\halg$. For example, let $\xi$ be a primitive fourth root of unity,  one can construct an example with two isomorphic fiber algebras corresponding to the matrices
\begin{equation*}
\cmmtx_1(\xi)=\begin{bmatrix} 0 & 1 & &\\-1 & 0 & & \\ && \ddots &\\&&&\ddots\end{bmatrix}\quad\text{and}\quad
\cmmtx_2(\xi)=\begin{bmatrix} 0 & 2 & &\\-2 & 0 & & \\ && 0&2\\&&-2&0\end{bmatrix}
\end{equation*}
as described in \prop{prop:c-ext} but nonisomorphic stabilizers.

In fact, two maximal ideals of $\calg$ with isomorphic stabilizers may not even be in the same orbit under $\autk(\halg,\calg)$. Let $V$ be an irreducible module of $\halg$ and $\rho:\halg\mapsto \halg/\mathrm{lann}(V)$ be the natural projection, then $\rho$ has a section of $\field$-algebra homomorphism in Case II but no such section in Case III. There are three orbits under $\autk(\halg,\calg)$: \{I\}, \{II,IV,V\} and \{III\}; where different choices of third primitive root of unity belong to the same orbit.
\label{ex:central-extension}
\end{example}
\begin{remark}
\hfill
\begin{enumerate}
\item
Unlike the examples in \cite[Section 5.1, 5.2]{MWY2023}, here $\{\stab{\maxi}\}_{\maxi\in\maxspec{\calg}}$ do not form a chain by containment. In fact, pick $\maxi_1$ and $\maxi_2$ in Cases II and III in \tab{tab:inv_mtx}; then $\stab{\maxi_1}\nsubseteq\stab{\maxi_2}$ and $\stab{\maxi_1}\nsupseteq\stab{\maxi_2}$.
\item
%For studying representations, one can take $\agp=\generate{\Im(\tccl)}$ w.l.o.g. and then $\orbidc\cong \cgp/\agp$.\item
%\par
%The group $\autk(\halg,\calg)$ acts transitively on elements in $\maxspec{\calg}$ having whose fiber algebras have irreducible modules of the same dimension but the winding automorphism groups do not. 
%In general the group $\gpfdh{\halg}$ does not act transitively on irreducible modules of higher dimensions by the action
%\begin{equation*}
%  \chi \cdot V=\chi\otimes V,\quad \forall\chi\in\gpfdh{\halg},\,V\in\irr(\halg)
%\end{equation*}
%even if the groups of their stabilizers are isomorphic, e.g. $V_1\in\irr(\halg/\maxi_1\halg)$ and $V_2\in\irr(\halg/\maxi_2\halg)$ for $\maxi_1$ and $\maxi_2$ defined in part(i) are not in the same orbit. Note that since in this section all irreducible modules are maximally stable, $\dim V=\sqrt{\abs{\stab{V}}}$). In fact, the annihilators of $V_1$ and $V_2$ are not even in the same orbit of $\autk(\halg,\calg)$ since the corresponding fiber algebras are not isomorphic. This follows form the observation that there is no subalgebra in $\halg/\maxi_2\halg$ isomorphic to the following subalgebra of $\halg/\maxi_1\halg$:
%\begin{equation*}
%  \frac{\field<x,y>}{(x^3-1,y^3-1,xy-\runit yx)}.
%\end{equation*}
%
In general the action of $\gpfdh{\halg}$ given in \eq{eq:G_0-action} is not transitive on a subset of $\maxspec{\calg}$ with isomorphic stabilizers as long as $\stab{V}$ has an element of order greater than $2$. This can be seen as follows: Suppose there is such an irreducible representation $V$ over $c=\ann{\calg}{V}$ and pick an irreducible module $V'$ over $c^{-1}$. Then $\dim V=\dim V'$. If there is some $\chi\in\gpfdh{\halg}$ satisfying $\chi\cdot V'=V$, then $\chi$ vanishes on $c^2$. But there is no character or one-dimensional irreducible module over $c^2$ by assumption and \prop{prop:c-ext}. This is a contradiction.% does not have basic fiber \cite[Theorem C(a)]{MWY2023}. %$c^{2}$ does not have basic fiber by \lm{lm:dim_onepair} as $c^2\neq 0$. This implies none of the irreducible modules of $\halg$ over $c^2$ is in $\gpfdh{\halg}$ \cite[Theorem C(a)]{MWY2023}.

\end{enumerate}
\end{remark}

\noindent{\bf{Acknowledgement}} The author is thankful to his advisor Prof. Milen Yakimov for posing the problem addressed in this paper as well as his guidance and encouragement. He also thanks his postdoctoral mentor Prof. Quanshui Wu for his support and advice and, Yimin Huang and Tiancheng Qi for helpful discussions.
\bibliographystyle{amsplain-nodash}
\bibliography{discriminant_ideal}
%\printbibliography[title=Test]
\end{document}